\documentclass{article}
\usepackage{amsmath,amssymb}
\usepackage{cite}
\newtheorem{theorem}{Theorem}[section]
\newtheorem{lemma}{Lemma}[section]
\newtheorem{cor}{corollary}[section]
\newtheorem{definition}{Definition}[section]
\newtheorem{proof}{Proof}
\newtheorem{remark}{Remark}

\title{The blow-up theorem of a discrete semilinear wave equation}
\author{Keisuke Matsuya}
\date{}
\begin{document}
\maketitle
\section{Introduction}
Consider the Cauchy problem for the semilinear wave equation
\begin{equation}\label{wave}
\begin{cases}
\displaystyle{\frac{\partial^2 u}{\partial t^2} = \Delta u + |u|^p}\ (p>1),\\
u(0,\vec{x})=f(\vec{x}),\\
\displaystyle{\frac{\partial u}{\partial t}(0,\vec{x}) = g(\vec{x})},
\end{cases}
\end{equation}
where $u:=u(t,\vec{x})\ (t\ge0,\ \vec{x}:=(x_1,\cdots,x_d)\in\mathbb{R}^d)$ and $\Delta$ is the $d$-dimensional Laplacian $\Delta := \sum\limits_{k=1}^{d}{\frac{\partial ^2}{\partial x_k^2}}$.
When the initial condition $f(\vec{x}),\ g(\vec{x})$ are continuous and unifomly bounded, there is a smooth solution for $t>0$ and whenever the solution is bounded.
It is well known that the solutions of this problem is not necessarily bounded.
For instance, considering the spatially uniform initial condition, $f(\vec{x}) \equiv 0,\ g(\vec{x})\equiv g>0$, this fact can be understood.
In this case,\ $u(t,\vec{x})=u(t)$ and \eqref{wave} becomes an ordinary differential equation,
\begin{equation}\label{ode1}
\begin{cases}
\displaystyle{\frac{d^2 u}{d t^2}=|u|^p}\\
u(0)=0\\
\displaystyle{\frac{d u}{d t}(0)=g>0}
\end{cases}.
\end{equation}
Because of the initial condition, the solution of \eqref{ode1} is non negative if it is bounded so that $|u|^p=u^p$ is obtained.
Multiplying the both sides by $\frac{du}{dt}$ and integrating 0 to $t$, we get
\begin{equation*}
\displaystyle{\left(\frac{d u}{d t}\right)^2=\frac{2}{p+1}u^{p+1}+g^2}.
\end{equation*}
Owing to $\frac{d^2 u}{d t^2} \ge 0$ and $\frac{d u}{d t}(0)=g>0$,\ $\frac{d u}{d t} \ge 0\ (t \ge 0)$ is derived.
Therefore the differential inequality,
\begin{equation}\label{ineq}
\displaystyle{\frac{d u}{d t}>\sqrt{\frac{2}{p+1}}u^{(p+1)/2}}
\end{equation}
is obtained.
Since there exists positive number $\varepsilon$ such that $u(\varepsilon)>0$,\ the solution of \eqref{ineq} is 
\begin{equation*}
\displaystyle{u(t)>\frac{(\alpha C)^{-1/\alpha}}{\left\{\left(\alpha C\right)^{-1}u(\varepsilon)^{-\alpha}+\varepsilon-t\right\}^{1/\alpha}}\quad(t>\varepsilon)}
\end{equation*}
where $\alpha=(p-1)/2$ and $C=\sqrt{2/(p+1)}$.
Now we see that the right side diverges as $t \to \alpha^{-1}C^{-1}u(\varepsilon)^{-\alpha}+\varepsilon-0$ so that the solution of \eqref{ode1} is not bounded for all $t \ge 0$.
In general, if there exists a finite time $T \in \mathbb{R}_{{}>0}$ and if the solution of \eqref{wave} in $(t,\vec{x})\in[0,T)\times\mathbb{R}^d$ satisfies
\begin{equation*}
\limsup\limits_{t\to T-0}{\| u(t,\cdot)\|_{L^{\infty}}}=\infty,
\end{equation*}
where
\begin{equation*}
\| u(t,\cdot) \|_{L^\infty}:=\sup\limits_{\vec{x}\in\mathbb{R}^d}{|u(t,\vec{x})|},
\end{equation*}
then we say that the solution of \eqref{wave} blows up at time $T$.
If such $T$ does not exist for the solution of \eqref{wave} then we call it a global solution.

The critical exponent $p_{\text{c}}(d):=\frac{d+1+\sqrt{d^2+10d-7}}{2(d-1)}\ (d\ge2)$ which characterises the blow up of the solutions for \eqref{wave} is studied by many researchers \cite{John,Glassey1,Glassey2,Schaeffer,Sideris,Georgiev,Yordanov}.
F. John\cite{John} proved small data blow up with $1<p<p_{\text{c}}(3)$ and small data global existence with $p_{\text{c}}(3)<p$.
R.T. Glassey\cite{Glassey1,Glassey2} proved small data blow up with $1<p<p_{\text{c}}(2)$ and small data global existence with $p_{\text{c}}(2)<p$.
J. Schaeffer\cite{Schaeffer} proved small data blow up with $p=p_{\text{c}}(d)$ where $d=2,3$.
T. Sideris\cite{Sideris} proved small data blow up with $1<p<p_{\text{c}}(d)$ where $d \ge 4$.
V. Georgiev, H. Lindblad and C. Sogge\cite{Georgiev} proved small data global existence with $p_{\text{c}(d)}<p$ where $d \ge 4$.
B. Yordanov and Q.S. Zhang\cite{Yordanov} proved small data blow up with $p=p_{\text{c}(d)}$ where $d \ge 4$.\\
Kato \cite{Kato} proved the following theorem
\begin{theorem}
Let $u$ be a generalized solution of
\begin{equation*}
\displaystyle{\frac{\partial^2 u}{\partial t^2} - \sum^d_{j,k=1}{\frac{\partial}{\partial x_j}a_{jk}(t,\vec{x})\frac{\partial}{\partial x_k}u}-\sum^d_{j=1}{\frac{\partial}{\partial x_j}a_j(t,\vec{x})u}}=f(t,\vec{x},u)\ (t \ge 0,\ \vec{x}\in\mathbb{R}^d)
\end{equation*}
on a time interval $0 \le t < T \le \infty$, which is supported on a forward cone
\begin{equation*}
K_R = \{(t,\vec{x});t \ge 0,\ |\vec{x}| \le t+R\}\ (R>0).
\end{equation*}
Assume that $f$ satisfies
\begin{equation*}
f(t,\vec{x},s) \ge
\begin{cases}
b|s|^{p_0}\ (|s| \le 1),\\
b|s|^p\ (|s| \ge 1),
\end{cases}
\end{equation*}
where $b>0$ and $1 < p \le p_0 = (d+1)/(d-1)$.\\
(If $d=1$,\ $p_0$ may be any number greater than or equal to $p$.)\\
Moreover, assume that, for $w(t)=\int_{\mathbb{R}^d}u(t,\vec{x})d\vec{x}$, either (a) $\frac{d w}{d t}(0) > 0$, or (b) $\frac{d w}{d t}(0) = 0$ and $w(0) = 0$.\\
Then one must have $T < \infty$.
\end{theorem}
From this theorem, we obtain the next corollary.
\begin{cor}\label{cor:1.1}
Let $u$ be the solution of \eqref{wave}.
Assume that $f$ and $g$ in \eqref{wave} satisfy $\text{supp}(f) \bigcup \text{supp}(g) \subset \{\vec{x} \in \mathbb{R}^d;|\vec{x}| \le K \}\ (K>0)$ and $\int_{\mathbb{R}^d}gd{\vec{x}} > 0$.
Moreover, assume $1 < p \le (d+1)/(d-1)\ (d \ge 2)$.\\
(If $d=1$, any assumptions for $p$ but $1<p$ are not needed.)
\\
Then $u$ blows up at some finite time.
\end{cor}

In numerical computation of \eqref{wave}, one has to discretize it and consider a partial difference equation.
A naive discretization would be to replace the $t$-differential and the Laplacian with central differences such that \eqref{wave} turns into
\begin{equation*}
\displaystyle{\frac{u^{\tau+1}_{\vec{n}}-2u^\tau_{\vec{n}}+u^{\tau-1}_{\vec{n}}}{\delta^2} = \sum\limits_{k=1}^d{\frac{u^\tau_{\vec{n}+\vec{e}_k}-2u^\tau_{\vec{n}}+u^\tau_{\vec{n}-\vec{e}_k}}{\xi^2}}+|u^\tau_{\vec{n}}|^p},
\end{equation*}
where $u(\tau,\vec{n})(=:u^\tau_{\vec{n}}):\mathbb{Z}_{{} \geq 0} \times \mathbb{Z}^d \to \mathbb{R}$, for positive constants $\delta$ and  $\xi$Cand where $\vec{e}_k \in \mathbb{Z}^d$\ is the unit vector whose $k$th component is $1$ and whose other components are $0$.
Putting $\lambda := \delta^2/\xi^2$, we obtain
\begin{equation}\label{dwave0}
u^{\tau+1}_{\vec{n}} = 2d\lambda\hat{M}(u^\tau_{\vec{n}})+(2-2d\lambda)u^\tau_{\vec{n}} - u^{\tau-1}_{\vec{n}} + \delta^2 |u^\tau_{\vec{n}}|^p  \quad (p > 1).
\end{equation}
Here
\begin{equation}\label{def:M}
\hat{M}(V_{\vec{n}}) := \displaystyle{\frac{1}{2d}}\sum\limits_{k=1}^{d}{(V_{\vec{n}+\vec{e}_k}+V_{\vec{n}-\vec{e}_k})}.
\end{equation}
For a spatially uniform initial condition,\ \eqref{dwave0} becomes an ordinary difference equation
\begin{equation*}
u^{\tau+1} = 2u^\tau - u^{\tau-1} + \delta |u^\tau|^p.
\end{equation*}
The above equation is a discretization of \eqref{ode1}, but the features of its solutions are quite different.
In fact, $u^\tau$ will never blow up at finite time steps.
Hence, \eqref{dwave0} does not preserve the global nature of the original semilinear wave equation \eqref{wave}.

In this article, we propose and investigate a discrete analogue of \eqref{wave} which does keep the characteristic of corollary \ref{cor:1.1}.\\
In section 2, we present a partial difference equation with a parameter $p$ whose continuous limit equals to  \eqref{wave}, and state the main theorem which shows that this difference equation has exactly the same properties as \eqref{wave} with respect to $p$.
This theorem is proved in section 3.
\section{Discretization of the semilinear wave equation}
We consider the following initial value problem for the partial difference equation
\begin{equation}\label{dwave}
\displaystyle{u^{\tau+1}_{\vec{n}}+u^{\tau-1}_{\vec{n}}=\frac{4v^\tau_{\vec{n}}}{2-\delta^2v^\tau_{\vec{n}}|v^\tau_{\vec{n}}|^{p-2}}},\quad (\tau \in \mathbb{Z}_{{} > 0},\ \vec{n} \in \mathbb{Z}^d)
\end{equation}
where $p>1$ and $\delta>0$ are parameters and $v^\tau_{\vec{n}}$ is defined by means of $\hat{M}$ \eqref{def:M} as
\begin{equation*}
v^\tau_{\vec{n}} := \hat{M}(u^\tau_{\vec{n}}).
\end{equation*}
If there exists a smooth function $u(t,\vec{x}) \quad (t \in \mathbb{R}_{{} \ge 0},\ \vec{x} \in \mathbb{R}^d)$ that satisfies $u(\tau\delta, \xi \vec{n}=u^\tau_{\vec{n}})$ with $\xi:=\sqrt{d}\delta$, we find
\begin{equation*}
u(t+\delta, \vec{x}) + u(t-\delta, \vec{x})= v(t,\vec{x}) ( 2 + \delta^2v(t,\vec{x})|v(t,\vec{x})|^{p-2}) + O(\delta^4),
\end{equation*}
with
\begin{equation*}
\displaystyle{v(t,\vec{x}) := \frac{1}{2d}\sum_{k=1}^d\left(u(t,\vec{x}+\xi\vec{e}_k)+u(t,\vec{x}-\xi\vec{e}_k)\right)},
\end{equation*}
or
\begin{align*}
\frac{u(t+\delta, \vec{x}) - 2u(t, \vec{x}) + u(t-\delta, \vec{x})}{\delta^2}&= \sum^d_{k=1}\frac{u(t, \vec{x}+\xi\vec{e}_k)-2u(t, \vec{x})+u(t, \vec{x}-\xi\vec{e}_k)}{\xi ^2}\\
&\qquad \qquad + |u(t, \vec{x})|^p + O(\delta^2).
\end{align*}
Taking the limit $\delta \to +0$, we obtain the semilinear wave equation \eqref{wave}
\begin{equation*}
\frac{\partial^2 u}{\partial t^2} = \Delta u + |u|^p.
\end{equation*}
Thus \eqref{dwave} can be regarded as a discrete analogue of \eqref{wave}.

Because of the term $2-\delta^2v^\tau_{\vec{n}}|v^\tau_{\vec{n}}|^{p-2}$, if $v^\tau_{\vec{n}} \to (2\delta^{-2})^{1/(p-1)}$, then $u^{\tau+1}_{\vec{n}} \to +\infty$.
This behaviour may be regarded as an analogue of th blow up of solutions for the semilinear wave equation.
Thus we define a blow up of solution for \eqref{dwave} as follow.
\begin{definition}
Let $u^\tau_{\vec{n}}$ be a solution of \eqref{dwave}.

When there exists $\tau_0 \in \mathbb{Z}_{{}\ge0}$ such that $v^{\tau}_{\vec{n}} \le (2\delta^{-2})^{1/(p-1)}$ for all $\tau < \tau_0$ and $\vec{n} \in \mathbb{Z}^d$, and there exists $\vec{n}_0 \in \mathbb{Z}^d$ such that $v^{\tau_0}_{\vec{n}_0} \ge (2\delta^{-2})^{1/(p-1)}$, then we say that the solution $u^\tau_{\vec{n}}$ blows up at time $\tau_0+1$.
\end{definition}
The example of blow-up solutions for \eqref{dwave} is as follow.
Considering the spatially uniform initial condition $u^0_{\vec{n}} \equiv 0,\ u^1_{\vec{n}} \equiv g>0$,\ $u^\tau_{\vec{n}}=u^\tau$ and \eqref{dwave} becomes an ordinary difference equation,
\begin{equation}\label{ode2}
\begin{cases}
\displaystyle{u^{\tau+1} + u^{\tau-1} = \frac{4u^\tau}{2-\delta^2u^\tau|u^\tau|^{p-2}}}\\
u^0 = 0\\
u^1 = g>0
\end{cases}
.
\end{equation}
This is the discrete analogue of \eqref{ode1}.
One can see that the solution of \eqref{ode2} blows up as follow.

Let the solution of \eqref{ode2} does not blow up at any $\tau$,\\
i.e., $u^\tau<(2\delta^{-2})^{1/(p-1)}\ ({}^\forall\tau\in\mathbb{Z}_{{}\ge0})$, then we get
\begin{equation*}
\displaystyle{u^{\tau+1} - 2u^\tau + u^{\tau-1} = \frac{2\delta^2|u^\tau|^p}{2-\delta^2u^\tau|u^\tau|^{p-2}}>0}.
\end{equation*}
Hence, we obtained a difference inequality $u^{\tau+1}-2u^\tau+u^{\tau-1}>0$.
Solving this inequality with the initial value,\ $u^\tau > g\tau$ is derived.
This inequality means that $u^\tau$ is arbitrarily large for large $\tau\in\mathbb{Z}_{{}>0}$.
This statement contradicts to $u^\tau<(2\delta^{-2})^{1/(p-1)}\ ({}^\forall\tau\in\mathbb{Z}_{{}\ge0})$ so that the solution of \eqref{ode2} blows up at some finite time.

Furthermore, \eqref{dwave} inherits quite similar properties to those of \eqref{wave}. The following theorem is the main result in this article.
\begin{theorem}
Let $u^\tau_{\vec{n}}$ be the solution for \eqref{dwave}.
Assume that 
\begin{enumerate}
\item[(A1)] $\{\vec{n} \in \mathbb{Z}^d;u^j_{\vec{n}} \neq 0\} \subset \{\vec{n} \in \mathbb{Z}^d;\|\vec{n}\| \le K\},\ (j=0,1\ K>0)$
\item[(A2)] $\sum_{\vec{n}}u^1_{\vec{n}}>\sum_{\vec{n}}u^0_{\vec{n}}$,
\end{enumerate}
where $\|\vec{n}\|:=|n_1|+\cdots+|n_d|\ (\vec{n}=(n_1,\cdots,n_d)\in\mathbb{Z}^d)$.\\
Moreover assume $1 < p \le (d+1)/(d-1)\ (d \ge 2)$.\\
(If $d=1$, any assumptions for $p$ but $1<p$ are not needed.)\\
Then $u^\tau_{\vec{n}}$ blows up at some finite time.
\end{theorem}
\begin{remark}
The summations in (A2) seem to be infinite series, but owing to (A1), both summations are finite series.\\
The author believes that \eqref{dwave} does keep the characteristic of the critical exponent $p_{\text{c}(d)}$.
\end{remark}
\section{Proof of the theorem}
The idea of the proof is similar to that adopted by Kato\cite{Kato}.

First, to make the equations simply we take the scaling $(2\delta^{-2})^{1/(p-1)}u^\tau_{\vec{n}} \to u^\tau_{\vec{n}}$ then \eqref{dwave} is changed to 
\begin{equation}\label{dwave2}
\displaystyle{u^{\tau+1}_{\vec{n}} + u^{\tau-1}_{\vec{n}} = \frac{2v^\tau_{\vec{n}}}{1-v^\tau_{\vec{n}}|v^\tau_{\vec{n}}|^{p-2}}}.
\end{equation}
We shall deduce a contradiction by assuming that $u^\tau_{\vec{n}}$ does not blow up at any finite time, i.e., $v^\tau_{\vec{n}}<1\ ({}^\forall(\tau,\vec{n}) \in \mathbb{Z}_{{}\ge0}\times\mathbb{Z}^d)$.

Put 
\begin{equation}\label{def:U}
U^\tau:=\sum_{\vec{n}}u^\tau_{\vec{n}}.
\end{equation}
Because of (A1), $\{\vec{n} \in \mathbb{Z}^d;u^\tau_{\vec{n}} \neq 0\} \subset \{\vec{n} \in \mathbb{Z}^d;\|\vec{n}\|\le K+\tau-1\}$ so that the summation of \eqref{def:U} is well-defined.
Moreover, from $\{\vec{n} \in \mathbb{Z}^d;v^\tau_{\vec{n}} \neq 0\} \subset \{\vec{n} \in \mathbb{Z}^d;\|\vec{n}\|\le K+\tau\}$, $U^\tau=\sum_{\vec{n}}{v^\tau_{\vec{n}}}$ and $v^\tau_{\vec{n}}<1$, we obtain the inequality as follow
\begin{equation}\label{inequality}
U^\tau < T^\tau,
\end{equation}
where $T^\tau:=\#\{\vec{n} \in \mathbb{Z}^d;\|\vec{n}\|\le K+\tau\}$.
From \eqref{dwave2}, we get
\begin{equation}\label{sum}
\displaystyle{\sum_{\vec{n}}{(u^{\tau+1}_{\vec{n}} - 2v^{\tau}_{\vec{n}} + u^{\tau-1}_{\vec{n}})} = \sum_{\vec{n}}{\frac{2|v^\tau_{\vec{n}}|}{1-v^\tau_{\vec{n}}|v^\tau_{\vec{n}}|^{p-2}}}}
\end{equation}
The left hand of \eqref{sum} is same to $U^{\tau+1}-2U^\tau+U^{\tau-1}$ and the right hand is nonnegative because that all terms of summation are nonnegative.\\
Hence we get $U^{\tau+1}-2U^\tau+U^{\tau-1} \ge 0$.
From this inequality, there exists some positive number $C_0$ which satisfies the inequality as follow
\begin{equation}\label{ineq2}
U^\tau \ge C_0 \tau
\end{equation}
for sufficiently large $\tau \in \mathbb{Z}_{{}>0}$.\\
Note that $U^\tau \ge 0$ for sufficiently large $\tau \in \mathbb{Z}_{{}>0}$.

To get another inequality about $U^\tau$, we need the next lemma.
\begin{lemma}
Put 
\begin{equation*}
\displaystyle{h(x)=\frac{2|x|^p}{1-x|x|^{p-2}}\ (x<1)}.
\end{equation*}
Let $0 \le x_0 < 1,\ x_{j-1} \le x_j\ (j=1,\cdots,s)$ and $\lambda_j \ge 0\ (j=0,\cdots,s),\ \lambda_0+\cdots+\lambda_s=1$.\\
If $\lambda_0x_0+\cdots+\lambda_sx_s \ge 0$ then the inequality as follow
\begin{equation*}
\lambda_0h(x_0)+\cdots+\lambda_sh(x_s) \ge h(\lambda_0x_0+\cdots+\lambda_sx_s)
\end{equation*}
is satisfied.
\end{lemma}
\begin{proof}
We get
\begin{align}
\displaystyle{\frac{\partial}{\partial x_0}(\lambda_0h(x_0)+\cdots+\lambda_sh(x_s) - h(\lambda_0x_0+\cdots+\lambda_sx_s))}\notag\\
= \lambda_0 (h^\prime(x_0)-h(\lambda_0x_0+\cdots+\lambda_sx_s)),\label{lem}
\end{align}
where $\displaystyle{h^\prime(x):=\frac{d h}{d x}(x)}$.\\
Since $h(x)$ is convex on the interval $[0,1)$, $h^\prime(x)$ increases monotonically on the interval $[0,1)$.
On the other hand, $0 \le \lambda_0x_0+\cdots+\lambda_sx_s \le x_0 < 1$ by the definitions.\\
Then we get that \eqref{lem} is nonnegative and 
\begin{align*}
\lambda_0&h(x_0)+\cdots+\lambda_sh(x_s) - h(\lambda_0x_0+\cdots+\lambda_sx_s)\\
&\ge \lambda_0h(-(\lambda_1x_1+\cdots\lambda_sx_s)/\lambda_0)+\lambda_1h(x_1)+\cdots+\lambda_sh(x_s)-h(0)\\
&\ge 0
\end{align*}
is obtained.\\
Now the proof of lemma is completed.\ $\blacksquare$
\end{proof}
Since $\{\vec{n} \in \mathbb{Z}^d;v^\tau_{\vec{n}} \neq 0\} \subset \{\vec{n} \in \mathbb{Z}^d;\|\vec{n}\|\le K+\tau\}$ and $U^\tau=\sum_{\vec{n}}v^\tau_{\vec{n}}$ is nonnegative for sufficiently large $\tau\in\mathbb{Z}_{{}\ge0}$, this lemma is adopted to right hand of \eqref{sum} as follow,
\begin{eqnarray*}
\displaystyle{\frac{2|v^\tau_{\vec{n}}|^p}{1-v^\tau_{\vec{n}}|v^\tau_{\vec{n}}|^{p-2}}}&\ge& \displaystyle{T^\tau\frac{2|\frac{1}{T^\tau}\sum_{\vec{n}}v^{\tau}_{\vec{n}}|^p}{1-\frac{1}{T^\tau}\sum_{\vec{n}}v^{\tau}_{\vec{n}}|\frac{1}{T^\tau}\sum_{\vec{n}}v^{\tau}_{\vec{n}}|^{p-2}}}\\
&=& \displaystyle{\frac{2(T^\tau)^{1-p}(U^{\tau})^p}{1-(T^\tau)^{1-p}(U^{\tau})^{p-1}}}.
\end{eqnarray*}
Here we put $\lambda_j=1/T^\tau\ (j=1,\cdots,T^\tau)$.\\
We note that there exists positive number $C_T$ which satisfies $T^\tau < C_T\tau^d$ for sufficiently large $\tau\in\mathbb{Z}_{{}>0}$.
Considering this statement and $U^\tau<T^\tau$, we get
\begin{equation*}
U^{\tau+1} - 2U^\tau + U^{\tau-1} \ge 2C_T^{1-p}\tau^{-d(p-1)}(U^\tau)^p
\end{equation*}
with sufficiently large $\tau\in\mathbb{Z}_{{}>0}$.\\
Since $1<p\le(d+1)/(d-1)$, i.e.,\ $-d(p-1) \ge -(p+1)$, we get
\begin{equation}\label{ineq3}
U^{\tau+1} - 2U^\tau + U^{\tau-1} \ge C_2\tau^{-(p+1)}(U^\tau)^p,
\end{equation}
where $C_2:=2C_T^{1-p}$.\\
Moreover, using \eqref{ineq},
\begin{equation*}
U^{\tau+1} - 2U^\tau + U^{\tau-1} \ge C_2C_0^{1-p}\tau^{-1}
\end{equation*}
 with sufficiently large $\tau\in\mathbb{Z}_{{}>0}$.\\
Solving this difference inequality, it is found that $U^\tau$ increases monotonically and there exists some positive number $C_1^\prime$ which satisfies inequality
\begin{equation}\label{ineq4}
U^\tau \ge C_1^\prime\tau\log{\tau},
\end{equation}
with sufficiently large $\tau\in\mathbb{Z}_{{}>0}$.\\
Now we consider about
\begin{equation*}
\displaystyle{E^\tau := (U^{\tau+1}-U^\tau)^2-\frac{C_2}{p+1}\tau^{-(p+1)}(U^\tau)^{p+1}}.
\end{equation*}
Since \eqref{ineq3} and $U^\tau$ is monotonically increasing, we get
\begin{align*}
E^{\tau+1}&-E^\tau \\
&=\{(U^{\tau+1}-U^{\tau})^2-(U^{\tau}-U^{\tau-1})^2\}\\
&~~~~~-\frac{C_2}{p+1}\{\tau^{-(p+1)}(U^{\tau})^{p+1} - (\tau-1)^{-(p+1)}(U^{\tau-1})^{p+1}\}\\
&\ge2\tau^{-(p+1)}(U^{\tau})^p(U^{\tau+1}-U^{\tau-1}) - \frac{C_2}{p+1}\tau^{-(p+1)}\{(U^{\tau})^{p+1}-(U^{\tau-1})^{p+1}\}\\
&\ge C_2\tau^{-(p+1)}(U^{\tau})^{p+1}\left\{1-\frac{U^{\tau-1}}{U^{\tau}}-\frac{1}{p+1}+\frac{1}{p+1}\left(\frac{U^{\tau-1}}{U^{\tau}}\right)^{p+1}\right\}.
\end{align*}
It is known that $\frac{1}{p+1}\lambda^{p+1}-\lambda+1-\frac{1}{p+1}>0\ (0\le\lambda\le1)$ so that we get $E^{\tau+1}-E^\tau>0$ with sufficiently large $\tau\in\mathbb{Z}_{{}>0}$.\\
Due to $U^\tau/\tau \ge C_1^\prime\log{\tau}$ by \eqref{ineq4}, there exists some positive number $C_3$ which satisfies
\begin{equation*}
\displaystyle{(U^{\tau+1}-U^\tau)^2 \ge C_3\tau^{-(p+1)}(U^\tau)^{p+1}}
\end{equation*}
with sufficiently large $\tau\in\mathbb{Z}_{{}>0}$.\\
Owing to \eqref{ineq4}, we get
\begin{align*}
U^{\tau+1} - U^\tau &\ge \displaystyle{C_3 \left(\frac{U^\tau}{\tau}\right)^{(p-1)/2}\frac{U^\tau}{\tau}}\\
&\ge \displaystyle{C_3C_1^{\prime(p-1)/2}(\log{\tau})^{(p-1)/2}\frac{U^\tau}{\tau}},
\end{align*}
with sufficiently large $\tau\in\mathbb{Z}_{{}>0}$.\\
Since $(\log{\tau})^{(p-1)/2}$ is arbitrarily large for large $\tau\in\mathbb{Z}_{{}>0}$, the following linear difference inequality
\begin{equation*}
\displaystyle{U^{\tau+1}-U^\tau \ge C \frac{U^\tau}{\tau}}
\end{equation*}
with any positive number $C$ and $\tau \ge {}^{\exists}\tau_0$ where $\tau_0$ depends on $C$ is held.\\
Solving this difference inequality, we get
\begin{equation*}
\displaystyle{U^\tau\ge\prod_{s=\tau_0}^{\tau-1}{\frac{s+C}{s}U^{\tau_0}}}\ (\tau\ge\tau_0+1).
\end{equation*}
Let $C>d+1$, then
\begin{equation*}
\displaystyle{U^\tau \ge U^{\tau_0}\prod_{k=0}^d{\frac{\tau+k}{\tau_0+k}}\ (\tau \ge \tau_0+1)}.
\end{equation*}
This inequality means that there exists some positive number $C^\prime$ which satisfies inequality $U^\tau \ge C^\prime\tau^{d+1}$ with sufficiently large $\tau\in\mathbb{Z}_{{}>0}$ but this statement contradicts to $U^\tau<T^\tau$.\\
Now the contradiction is deduced and the proof of the theorem is completed.
\section*{Acknowledgement}
I am deeply grateful to Prof. Tokihiro who provided helpful comments and suggestions. 

\end{document}